\documentclass[preprint,10pt]{elsarticle}      

\usepackage{lipsum}
\usepackage{longtable}
\usepackage{booktabs}
\usepackage{stmaryrd}
\usepackage{url}
\usepackage{longtable}
\usepackage{array}
\usepackage{xspace}
\usepackage[all=normal,paragraphs,floats,bibnotes,wordspacing,charwidths,leading,indent]{savetrees}

\usepackage[english]{babel} 
\usepackage{verbatim}
\usepackage{color,soul}

\usepackage{graphicx}
\usepackage{epsfig}
\graphicspath{{metric_paper/figures/}{}}
\usepackage{placeins} 
\usepackage[figuresright]{rotating}
\RequirePackage[caption=false,position=top]{subfig}

\newsubfloat{figure}
\usepackage{multirow}

\usepackage{savesym}
\usepackage{amsmath}
\savesymbol{iint,openbox}
\usepackage{txfonts}
\usepackage{pxfonts}
\restoresymbol{TXF}{iint}
\restoresymbol{TXF}{openbox}

\usepackage{amsfonts}
\usepackage{bbm}
\usepackage{amssymb}
\usepackage{dsfont}
\usepackage{mathtools}
\usepackage{etoolbox}
\usepackage{accents}

\usepackage{algorithmic} 
\usepackage[Algorithme]{algorithm} 

\usepackage{lipsum}

\let\origdoublepage\cleardoublepage
\renewcommand{\cleardoublepage}{%
  \clearpage{\pagestyle{empty}\origdoublepage}}

\makeatletter
\newcommand{\ChapterOutsidePart}{%
   \def\toclevel@chapter{-1}\def\toclevel@section{0}\def\toclevel@subsection{1}}
\newcommand{\ChapterInsidePart}{%
   \def\toclevel@chapter{0}\def\toclevel@section{1}\def\toclevel@subsection{2}}
\makeatother

\makeindex

\usepackage{citeref}
\renewcommand{\bibitempages}[1]{\newblock {\scriptsize [\mbox{cited at p.\ }#1]}}
\usepackage{color}

\newtheorem{theorem}{Theorem}
\newtheorem{lemma}[theorem]{Lemma}

\newtheorem{proposition}[theorem]{Proposition}
\newtheorem{assumption}[theorem]{Hypothesis}

\newtheorem{remark}[theorem]{Remark}

\newproof{proof}{Proof}

\newcommand{\R}{\mathbb{R}}

\newcommand{\N}{\mathbb{N}}
\newcommand{\Rnd}{\R^{n \times d}}

\newcommand{\NN}{{\mathcal{N}}}
\newcommand{\KK}{{\mathcal{K}}}

\newcommand{\EE}{\mathcal{E}}

\newcommand{\UU}{\mathcal{U}}
\renewcommand{\SS}{\mathcal{S}}

\newcommand{\car}[1]{\chiup\left({#1}\right)} 

\newcommand{\ie}{{\it i.e. }}

\newcommand{\eg}{{\it e.g. }}

\newcommand{\uak}{\alpha_k}

\newcommand{\tak}[1]{\tilde{\tau}_{#1}}

\newcommand{\ps}[2]{\langle #1|#2 \rangle}

\newcommand{\ssum}{\displaystyle\sum}

\newcommand{\pprod}{\displaystyle\prod}

\newcommand{\mmin}{\displaystyle\min}
\newcommand{\mmax}{\displaystyle\max}

\newcommand{\taa}{\tilde{a}}
\newcommand{\tA}{\tilde{A}}

\newcommand{\aaa}{\alpha}

\newcommand{\tE}{\tilde{\EE}}

\newcommand{\zeroDmun}{\{0,\ldots,D-1\}}
\newcommand{\Diam}{\Delta_\NN}
\newcommand{\tT}{\tilde{T}}
\newcommand{\tc}{\tilde{c}}

\newcommand{\cor}{\ifmmode coeur \else c\oe{}ur\fi\xspace}

\newcommand{\btau}{{\bf \tauup}}
\newcommand{\tbtau}{{\bf \tilde{\tauup}}}
\newcommand{\deriv}[2]{\frac{\partial #1}{\partial #2}}

\let\overfence\overbrace 
\let\downfencefill\downbracefill 
\patchcmd{\overfence}{\downbracefill}{\downfencefill}{}{}
\patchcmd{\downfencefill}{\braceru \bracelu}{}{}{}

\begin{document}

\pagestyle{empty} 


\title{Multi-agent flocking under topological interactions}

\author[add1]{Samuel Martin}
\ead{samuel.martin@univ-lorraine.fr}

\address[add1]{S. Martin was with the Laboratoire Jean Kuntzmann, Universit\'e de Grenoble,
        B.P. 53, 38041 Grenoble Cedex 9, France during the work. S. Martin is now with the Centre de Recherche en Automatique de Nancy, Universit\'e de Lorraine,
        2 av de la for\^et de Haye, 54501 Vandoeuvre-l\`es-Nancy, France.This research has been partially supported by the ANR (project VEDECY).}

\maketitle



\pagestyle{empty} 

\begin{abstract}
In this paper, we consider a multi-agent system consisting of mobile agents with second-order dynamics.
The communication network is determined by the so-called topological interaction rule: agents interact with a fixed number of their closest neighbors. This rule comes from observations of real flock of starlings. The non-symmetry of the interactions adds to the difficulty of the analysis.
The goal of this paper is to determine practical conditions (on the initial positions and velocities of agents) ensuring that
the agents asymptotically agree on a common velocity (\ie  a flocking behavior is achieved).
For this purpose, we define a notion of hierarchical structure in the interaction graph which allows to establish such conditions, building upon previous work on multi-agent systems with switching communication networks. Though conservative, our approach gives conditions that can be verified a priori. Our result is illustrated through simulations. 
\end{abstract}

\section{Introduction}

Analysis and design of cooperative behaviors in networked dynamic systems has lately received a lot of attention.
Multi-agent systems find applications in technical areas such as mobile sensor networks~\cite{Curtin2009}, cooperative robotics~\cite{Cortes2004}
or distributed implementation of algorithms~\cite{Tsitsiklis1986}.
A central question arising in the study of multi-agent systems is whether the group will be able to reach a consensus. Intuitively, agents are said to reach a consensus 
when all individuals agree on a common value (e.g. the heading direction of a flock of birds, the candidate to elect for voters). 

To carry out formal studies on consensus problems, one usually assumes that the multi-agent system follows some abstract communication protocol and then investigates conditions under which a consensus will be reached. Existing frameworks include discrete and continuous-time systems, involving or neglecting time-delays in the communication process. The communication network between agents is usually modeled by a graph. Its topology is either assumed to be fixed, or can switch over time. The switching topology of the interactions is sometimes assumed to depend on the state of the agents (e.g. the strength of the communication can be a function of the distance between agents). The order of the dynamics of the agents also varies between the different models. For example, second-order models can be useful to represent the dynamics of both the speed and position of agents. Olfati-Saber, Fax and Murray review results on the subject in~\cite{Saber2007}.

Most papers have investigated sufficient conditions ensuring asymptotic consensus.  The assumptions made in the models are usually rather general (see \eg\cite{Moreau2005}). This enables the given conditions to apply in a wide range of cases. Conditions usually require invariant connectivity properties in the communication network over time. A drawback in such conditions is that they often cannot be verified a priori. Our approach differs since we consider a group of agents with second-order dynamics where communication between agents depends on their state. The goal of this paper is to determine practical conditions (on the initial positions and velocities of agents) ensuring that
the agents eventually agree on a common velocity (\ie  a flocking behavior is achieved).

In most of the literature on flocking, \eg~\cite{Reynolds1987,Vicsek1995,Jadbabaie2003,Saber2007,CuckerSmale2007TAC}, researchers have assumed symmetric interactions. Agents interact within a certain communication radius~\cite{Martin2010, MartinFazeli2012, Martin2013}. Such symmetric interactions ease the analysis of the system. However, a recent field study of a starling flock suggested that so-called topological interactions - agents interact with a fixed number of their closest neighbors - reproduce more accurately the collective behaviors observed in nature~\cite{Ballerini2008}. Such interactions also increase the robustness of the flock against predator attacks. Our aim is to formally analyze the consequence of these interactions on velocity alignment. The non-symmetry of this type of interactions adds to the difficulty of the analysis. To tackle the issue, we define a notion of hierarchical structure in the interaction graph, building upon previous work such as~\cite{Angeli2009,Moreau2004}, and then study the robustness of such structure adapting ideas from~\cite{Martin2010}. This allows us to establish new practical conditions for flocking.
Though conservative, our approach gives conditions that can be verified a priori. Moreover, it is computationally tractable and can be fully automated. Our result is illustrated through simulations. 

\subsection{Problem Formulation}\label{sec:pb-form}

In this paper, we study a continuous-time, multi-agent system. We consider a set $\NN=\{1,\dots,n\}$ of mobile {\it agents}  evolving in a $d$-dimensional space. Each agent $i\in \NN$ is characterized by its position $x_i(t)\in \R^d$ and its velocity $v_i(t)\in \R^d$.
The initial positions and velocities are given by $x_i(0)=x_i^0$ and $v_i(0)=v_i^0$. The agents exchange information over a {\it communication network} given by a graph $G(t)=(\NN,\EE(t))$; the topology of the communication network depends on the relative position of agents and is therefore subject to change.
The agents use the available information to adapt their velocity in order to achieve a flocking behavior. 
Formally, the evolution of each agent $i\in \NN$ is described by the following system of differential equations:
\begin{equation}\label{sys:alignement-en-vitesse-topo}
\begin{array}{l}
\dot{x_i}(t) = v_i(t), \\
\dot{v_i}(t) = \ssum_{j=1}^n a_{ij}(t) (v_j(t) - v_i(t)),
\end{array}
\end{equation}
where $a_{ij}(t)=1$ if $j$ belongs to the $m$ closest neighbors of $i$ at time $t$, \ie
$$
a_{ij} = \car{\left|\left\{\; k\in \NN \;| \; \|x_i-x_j\| > \| x_i - x_k \| \; \right\}\right| < m} \; ,
$$
where $m$ is a constant parameter depending on the model, and $\car{A}=1$ if statement $A$ is true and $0$ otherwise.

In this system, the weights $a_{ij}$ depend on the distance $\|x_i-x_j\|$ compared to the other distances $\|x_i-x_k\|$ for $k \in \NN \setminus \{i,j\}$. These interactions are generally non-symmetric. Such interactions are termed {\it topological} interaction due to the fact that they depend on topological distance of the graph associated the communication network rather than to Euclidean  distances. Another property of the communication network is that its associated graph is $m$-regular, \ie the in-degree of each agent is constant, equal to $m$. The aim of the present study is to find practical conditions on $x^0_i$ and $v^0_i$ for $i \in \NN$ such that velocity alignment is achieved, \ie there exists some constant velocity $v^*$ such that
$$
\forall i\in \NN, \lim_{t\longrightarrow +\infty} v_i(t) = v^*.
$$

In the rest of the paper, we start by presenting the approach we have used, then we explicit the main result on velocity alignment (flocking) and we end with an illustration of our result through simulations.

\subsection{Approach}\label{sec:approach}

In order to show that the trajectory of the system (\ref{sys:alignement-en-vitesse-topo}) converges toward velocity alignment, we study the robustness of some structure of the interaction graph when agents' positions are subject to disturbances. This enables to show that a {\it spanning tree}~\footnote{A spanning tree in a graph $G=(\NN,\EE)$ is a graph $T=(\NN,\tE)$ such that $\tE \subseteq \EE$ and there exists a node $r\in \NN$ called root of $T$ such that all other nodes $i\in \NN\setminus \{r\}$ are reachable from $r$. A node $i$ is reachable from $r$ if there is a {\it path} $(i_0,\ldots,i_p)$ in $T$ such that $i_0 = r$, $i_p = i$ and for all $k \in \{0,\ldots,p-1\}$, $(i_k,i_{k+1}) \in T$.} is preserved in the interaction graph over time. This property leads to the contraction of the velocities toward consensus.

Precisely, we proceed using the following reasoning:
\begin{description}
\item (i) The preservation of a spanning tree in the interaction graph guarantees the velocity alignment with an exponential rate (Theorem~\ref{th:flocking-topo-contraction}). This rate depends on the hierarchical structure which is induced by the preserved spanning tree (see section~\ref{sec:flocking-topo-structure-hierarchique}).
\item (ii) Using what precedes, an integration of the velocities allows to estimate the disturbance on the agents' distances. A robustness analysis then guarantees the preservation of the spanning tree required to obtained the velocity alignment (Lemma~\ref{lemma:preserv}).
\item (iii) Combining the two previous observations gives a condition on the initial position and velocities under which the system converges toward velocity alignment (Theorem~\ref{th:flocking-topo-main}).
\end{description}

To provide the convergence rate for the velocity alignment (Theorem~\ref{th:flocking-topo-contraction}) we adapt ideas from~\cite{Moreau2004} and~\cite{Angeli2009}. In~\cite{Moreau2004}, Moreau shows that the trajectory of the system converges toward a consensus provided a general connectivity assumption. However, the generality of his result prevented from obtaining a contraction rate which we need here. In~\cite{Angeli2009}, Angeli and Bliman have determined an asymptotic contraction rate for a discrete-time system analogous to the continuous-time system which we analize here. However, this result cannot be directly used in our case : first, we study a continuous-time system, and second and most importantly, our approach is based on a convergence rate valid for all time.

The non-symmetry of the system increases the difficulty of the analysis. As a consequence, obtaining a convergence rate toward velocity alignment in the present setting constitutes the major contribution of the present paper. We detail the derivation to obtain the convergence rate in sections~\ref{sec:flocking-topo-contraction-du-diametre} and~\ref{sec:flocking-topo-preuves}. One reason why the non-symmetry of the system adds to the difficulty of the analysis is that the average velocity $v^*$ is not preserved over time. Thus it is not possible to use the algebraic approach (see for instance~\cite{Martin2010,MartinFazeli2012,Martin2013}) as it has been done
in the symmetric case. In the present non-symmetric case, the function $\|\delta\|^2/2$, with $\delta = v-v^*$ being the velocity disagreement vector, is not a Lyapunov function anymore. Therefore, we have to turn to another Lyapunov function better adapted to non-symmetric interactions: the velocity diameter
$$
\Delta_\NN(t) = \mmax_{i,j \in \NN} \|v_i(t)-v_j(t)\|.
$$
We now present the 3-step approach.

\section{Sufficient conditions for flocking}

\subsection{Hierarchical structure}\label{sec:flocking-topo-structure-hierarchique}

We start by giving the notation required to introduce the hierarchical structure induced by a spanning tree in the interaction graph.
Let us consider a graph $H$ with node set $\NN$ and adjacency matrix $\tA = (\taa_{ij})$. 
We assume that $H$ has a spanning tree. Denote $r$ its root and $D$ its depth. The graph $H$ will play the role of the subgraph we aim at preserving in the interaction graph.

Following the idea from Angeli and Bliman, we define two sequences of $D+1$ subsets of nodes in $\NN$. For $k \in \{0,\ldots,D\}$, $\SS_k$ contains nodes being at distance $k$ of the root $r$ and set $\UU_k$ is the union of $\SS_l$ for $l\le k$ :
\begin{itemize}
\item $\SS_0 = \{r\}$,
\item $\forall k \in \{0,\ldots,D\}, \UU_k = \displaystyle\cup_{l=0}^k \SS_l$,
\item $\forall k \in \{0,\ldots,D-1\}, \SS_{k+1} = \NN^+_{\SS_k} \backslash \UU_k$,
\end{itemize}
where $\NN^+_{S_k}$ is the set of out-neighbors of the node set $\SS_k$ in the graph $H$, \ie
$$\NN^+_{S_k} = \{i\in \NN | \exists j \in S_k, \taa_{ij} = 1 \}.$$
As defined, sequence $(\SS_k)_{k\in \{0,\ldots,D\}}$ is a partition of node set $\NN$ and $(\UU_k)_{k\in \{0,\ldots,D\}}$ is an increasing family satisfying $\UU_{D} = \NN$. For $k \in \{1,\ldots,D-1\}$, we lower bound the sum of interaction weights from $\UU_k$ to $ \UU_k$ in $H$:
\begin{equation}
\uak = \displaystyle \min_{i \in \UU_{k+1}} \displaystyle\sum_{j\in \UU_k} \taa_{ij}.
\end{equation}
We give an example of hierarchical structure for a topological communication network in Figure~\ref{fig:structure-hierarchique}.

\begin{figure}[!htbp]
\begin{center}
\includegraphics[scale=0.6,trim=4cm 9cm 4cm 9cm,keepaspectratio]{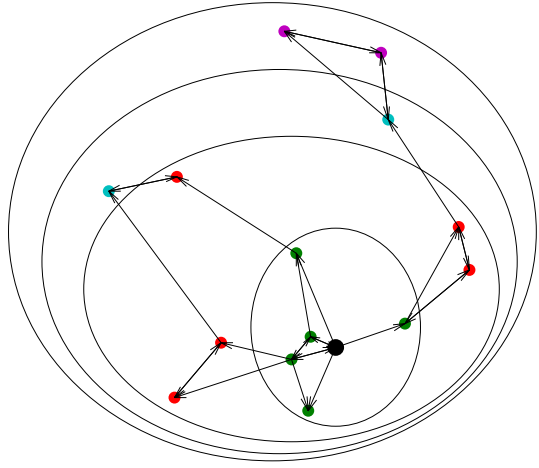}
\end{center}
\caption{\label{fig:structure-hierarchique} Color dots represent agents' positions. Arrows represent influences according to topological interactions for $m=2$. The bigger black dot is the root $r$. Black circles correspond to the increasing sequence of sets $\UU_k$ for $k \in \{1,\ldots,D\}$. Set $\UU_0 = \{r\}$ is not displayed. In the present case, $D = 4$ and $\alpha_1 = \alpha_2 = \alpha_3 = 1$.}
\end{figure}

This notation will serve to present our main contribution. To do so, we also need to define the following functions of $\R^+$ :
$$
c_0 = \tau\longmapsto \frac{1}{m+1} (1 - e^{-(m+1)\tau}) 
$$
and for $k \in \{1,\ldots,D-1\}$,
$$
c_k = \tau\longmapsto e^{-m \tau}(e^{\uak\tau}-1).
$$
This allows us to define functions $c : \R^D\longrightarrow \R^+$ and $T : \R^D\longrightarrow \R^+$ such that
$$
c = \btau \longmapsto \pprod_{k=0}^{D-1}c_k(\tau_k) \;\; \text{ and } \;\; T = \btau \longmapsto \ssum_{k=0}^{D-1}\tau_{k},
$$
where $\btau = (\tau_0,\ldots,\tau_{D-1})$. We choose a sequence $\tbtau = (\tak{0},\tak{1},\ldots,\tak{D-1})$ of non-negative real numbers which maximizes ratio $c/T$. If such a sequence is not unique, we choose one arbitrarily. Notice that such a sequence always exists since the function to be optimized is continuous over $(\R^+)^D$ and has a $0$ limit when $\|\btau\|$ goes to $+\infty$ (which shows that the function is bounded). To finish with, we define
\begin{equation}\label{eq:tildeT}
\tT = T(\tbtau) \;\; \text{ and } \;\;\tc = c(\tbtau),
\end{equation}
so that ratio $\tc/\tT$ is the maximal value of ratio function $c/T$. The reason why we define these notation will become clear in section~\ref{sec:flocking-topo-contraction-du-diametre}. Moreover, we give a method in section~\ref{sec:flocking-topo-optimisation-de-la-borne} to compute quantities $\tc$ and $\tT$.

Consider an interaction graph $G$ of adjacency matrix $A = (a_{ij})$.
The robustness analysis in section~\ref{sec:flocking-topo-robustesse} aims at obtaining the following hypothesis:
\begin{assumption}[Preservation of the hierarchical structure]\label{as:alpha-preserv}
The interaction graph $G$ satisfies the two following properties:
\begin{itemize}
\item for $i \in \UU_{1}$, $a_{ir} = 1$.
\item for $k \in \{1,\ldots,D-1\}$ and $i \in \UU_{k+1}$,
$$
\ssum_{j\in \UU_k} a_{ij} \; \ge \; \uak \;\; \text{ and } \;\; \ssum_{j\notin \UU_k} a_{ij} \; \le \; (m-\uak).
$$
\end{itemize}
\end{assumption}

In order to obtain the previous hypothesis, we have the following proposition:
\begin{proposition}\label{prop:topo-alphak}
Let $G$ be a $m$-regular interaction graph having the spanning tree $H$ as described above. Then, hypothesis~\ref{as:alpha-preserv} is satisfied.
\end{proposition}
This proposition can be deduced directly from the definition of $\uak$, from $G$ being $m$-regular and from $H \subseteq G$.

\subsection{Main result}\label{sec:flocking-topo-resultat-principal}

We consider a quantity $\rho \in [0,\Delta_\NN(0)]$ which represents the maximal disturbance authorized on the distances between agents (see section~\ref{sec:flocking-topo-robustesse}). We defined subgraph $H_\rho = (\NN,\tE,\tA)\subseteq G_{x^0}$ where $\tA = (\taa_{ij})$ with
\begin{equation}\label{eq:H-interaction-topo-perturbee}
\taa_{ij} = \car{\left| \left\{\; k\in \NN \;|\; \|x_i-x_j\| > \| x_i - x_k \|-2\rho \; \right\}\right| < m }.
\end{equation}
Graph $H_\rho$ corresponds to the subgraph of $G_{x^0}$ which is always preserved for any disturbance on the distances between agents provided that the disturbance remains smaller than $\rho$ (see lemma~\ref{lemma:preserv}).
In our result on velocity alignment, we assume that $H_\rho$ has a spanning tree and thus, {\it a fortiori}, so has $G_{x^0}$. We detail in section~\ref{sec:flocking-topo-optimisation-de-la-borne} how to chose $\rho$ so that this property is satisfied. We can use the hierarchical structure notation given in section~\ref{sec:flocking-topo-structure-hierarchique} by setting $H := H_\rho$. Under this condition, the following result is satisfied.

\begin{theorem}\label{th:flocking-topo-main}
Let $x^0=(x^0_1,\dots,x^0_n) \in \R^{nd}$ and $v^0=(v^0_1,\dots,v^0_n) \in \R^{nd}$ be the stacked vectors of positions and velocities, respectively. Let $\rho\in [0,\Diam(0)]$ so that $H_\rho$ has a spanning tree. Assume that the initial velocity diameter satisfies
\begin{equation}
\Delta_\NN(0) \leq \frac{\tilde{c}}{\tilde{T}} \rho.
\end{equation}
where $\tilde{c}$ and $\tilde{T}$ are defined using equation (\ref{eq:tildeT}).
Then, for all trajectory $(x(t),v(t))$ of system (\ref{sys:alignement-en-vitesse-topo}) defined over $\R^+$, $H_\rho$ is preserved in interaction graph $G(t)$ for all time $t\ge 0$ and all agents asymptotically converge toward velocity alignment.
\end{theorem}

This theorem shows that the bigger the authorized disturbance $\rho$, the bigger the authorized initial velocity diameter $\Diam(0)$ is. However, the ratio $\frac{\tilde{c}}{\tilde{T}}$ is a non-increasing function of $\rho$ as we shall show in section~\ref{sec:flocking-topo-optimisation-de-la-borne}. We will discuss in this section how to choose $\rho$ so as to optimize the bound $\frac{\tilde{c}\rho}{\tilde{T}}$. The proof of the theorem requires several intermediate results, for instance to characterize the convergence rate of the diameter (see theorem~\ref{th:flocking-topo-contraction}). Consequently, we transfer the proof to the end of section~\ref{sec:flocking-topo-preuves}.

\subsection{Contraction of the diameter}\label{sec:flocking-topo-contraction-du-diametre}

In this section, we show that the preservation of a hierarchical structure as detailed in section~\ref{sec:flocking-topo-structure-hierarchique} allows to bound below the contraction rate of the diameter.

Consider the following consensus system:
\begin{equation}\label{sys:consensus-v}
	\dot{v_i} = \ssum_{j=1}^n a_{ij} (v_j - v_i), i\in \NN,
\end{equation}
where $v_i \in \R^d$, for $j \in \NN$, $a_{ij} \in \{0,1\}$ and $\sum_{j\in \NN} a_{ij} = m$ is a given constant parameter. The idea which will be used to analyze this system are adapted from~\cite{Moreau2004} and~\cite{Angeli2009}. The contraction of the velocity diameter requires the preservation of the hierarchical structure of the interaction graph. 
We use notation $r$, $D$, $\SS_k$, $\UU_k$ and $\aaa_k$, etc. as defined in section~\ref{sec:flocking-topo-structure-hierarchique}. 
For a subset $S\subseteq \NN$, denote
$$
\Delta_S(t) = \mmax_{i,j \in S} \|v_i(t) - v_j(t)\|
$$
the velocity diameter of $S$ at time $t\ge 0$.
Then, we introduce the lemma which allows to characterize $\Delta_{\UU_{k+1}}$, diameter of $\UU_{k+1}$, in function of $\Delta_{\UU_{k}}$, diameter of $\UU_{k}$, for $k \in \zeroDmun$. An induction on this lemma will allow us to obtain a contraction rate for $\Delta_\NN$.

\begin{lemma}\label{l:deltauk1ttau-principal}
Let $k$ in $\zeroDmun$. Let times $t_0$ and $t$ such that $t_0<t$ and $\tau \geq 0$. Assume that the interaction graph $G(s)$ satisfies hypothesis~\ref{as:alpha-preserv} for $s \in [t,t+\tau]$. Then, we have
\begin{equation}\label{eq:rec}
\Delta_{\UU_{k+1}}(t+\tau) \leq \Delta_\NN(t_0) - c_k(\tau) (\Delta_{\NN}(t_0) - \Delta_{\UU_k}(t))
\end{equation}
where the $c_k(\tau)$ were defined in section~\ref{sec:flocking-topo-structure-hierarchique}.
\end{lemma}
This lemma is the core of our result. Its proof is given in section~\ref{sec:flocking-topo-preuves}. In order to give the main result of the section, we make use of the sequence of time intervals $(\tak{0},\tak{1},\ldots,\tak{D-1})$ as well as notation $\tT$ et $\tc$ as defined in section~\ref{sec:flocking-topo-structure-hierarchique} :
\begin{theorem}[Contraction of the velocity diameter]\label{th:flocking-topo-contraction}
Consider $v$ a trajectory of system (\ref{sys:consensus-v}) defined on $\R^+$ and assume that $G(t)$ satisfies hypothesis~\ref{as:alpha-preserv} (preservation of the hierarchical structure) over time interval $[0,Q\tT]$ for $Q \in \N$. Then, for all $q \in \{0,\ldots,Q \}$ 
we have
\begin{equation}\label{eq-gen-bound}
	\Delta_\NN(q\tT) \leq (1-\tc)^q \Delta(0).
\end{equation}
\end{theorem}
\begin{remark}
As a direct corollary of theorem~\ref{th:flocking-topo-contraction}, we obtain that the velocity diameter decreases exponentially fast toward $0$.
\end{remark}

\begin{proof}
The core of the theorem lies in Lemma~\ref{l:deltauk1ttau-principal} whose proof is given in section~\ref{sec:flocking-topo-preuves}. The present proof is adapted from~\cite{Angeli2009}. We show by induction the following statement: for all $t_0$ in $\R^+$, for $k\in \{1,\ldots,D-1\}$,
\begin{equation*}
P_k \Leftrightarrow \Delta_{\UU_{k+1}}(t_0+\ssum_{h=0}^k\tak{h}) \leq \left(1-\pprod_{h=0}^{k}c_h(\tak{h})\right) \Delta_{\NN}(t_0).
\end{equation*}
First, let us remark that Lemma~\ref{l:deltauk1ttau-principal} applied to $k=0$ and $\tau = \tak{0}$ gives
\begin{equation*}
\Delta_{\UU_{1}}(t_0+\tak{0}) \leq \Delta_\NN(t_0) - c_0(\tak{0}) (\Delta_{\NN}(t_0) - \Delta_{\UU_0}(t_0)).
\end{equation*}
Since $\UU_0=\{r\}$, $\Delta_{\UU_0}=0$ and the previous equation becomes
\begin{equation*}
\Delta_{\UU_{1}}(t_0+\tak{0}) \leq (1 - c_0(\tak{0})) (\Delta_{\NN}(t_0)),
\end{equation*}
which starts the induction.
Now, assume that for $k \in \{1,\ldots,D-2\}$,
\begin{equation*}
\Delta_{\UU_{k}}(t_0+\ssum_{h=0}^{k-1}\tak{h}) \leq \left(1 - \pprod_{h=0}^{k-1}c_h(\tak{h})\right) \Delta_{\NN}(t_0).
\end{equation*}
Once again, using Lemma~\ref{l:deltauk1ttau-principal} applied to $k$ and $\tau = \tak{k}$,
\begin{eqnarray*}
&&\Delta_{\UU_{k+1}}(t_0+\ssum_{h=0}^k\tak{h}) \;\;\leq\;\; \Delta_\NN(t_0) - c_k(\tak{k}) \left(\Delta_{\NN}(t_0) - \Delta_{\UU_k}(t_0+\ssum_{h=0}^{k-1}\tak{h})\right) \\
&&\leq \Delta_\NN(t_0) - c_k(\tak{k}) \left(\Delta_{\NN}(t_0) - (1 - \pprod_{h=0}^{k-1}c_h(\tak{h})) (\Delta_{\NN}(t_0))\right) \\
&&\leq \Delta_\NN(t_0) - c_k(\tak{k}) \left((\pprod_{h=0}^{k-1}c_h(\tak{h})) \Delta_{\NN}(t_0)\right) 
\;\;\leq\;\; \left(1-\pprod_{h=0}^{k}c_h(\tak{h})\right) \Delta_{\NN}(t_0),
\end{eqnarray*}
which ends the proof of the induction. We use the induction statement with $k=D-1$ to obtain
\begin{equation*}
\Delta_\NN(t_0+\tilde{T}) \leq (1-\tilde{c}) \Delta_\NN(t_0).
\end{equation*}
We obtain the result of the theorem by repeating the inequality.
\end{proof}

\subsection{Robustness of the hierarchical structure}\label{sec:flocking-topo-robustesse}

First, we start by showing that if the disturbance applied to the relative distances is not greater than $\rho$ then the subgraph $H_\rho$ defined in section~\ref{sec:flocking-topo-structure-hierarchique} is preserved. We again use notation given in section~\ref{sec:flocking-topo-structure-hierarchique}.
\begin{lemma}[Preservation of the subgraph]\label{lemma:preserv}
Consider a reference position vector $x^0 \in \Rnd$ and a disturbed position vector $y$ satisfying for all $i,j \in \NN$
$$
\| y_j - y_i - (x_j^0 - x_i^0)\| \le \rho.
$$
Then, graph $H_\rho$ defined by equation (\ref{eq:H-interaction-topo-perturbee}) is preserved in the topological interaction graph $G_y$.
\end{lemma}
The constraint we assign to the disturbed positions implies that the modification on the relative distances between agents must not be greater than $\rho$.

\begin{proof}
Let $i$ and $j$ in $\NN$ such that edge $(j,i)$ belongs to graph $H_\rho$ (\ie $\taa_{ij}=1$).
Denote $V_{x^0}$ the set $\{ k\in \NN \;|\; \|x_i^0-x_j^0\| > \| x_i^0 - x_k^0 \|-2\rho\}$ and $V_y$ the set $\{ k\in \NN \;|\; \|y_i-y_j\| > \| y_i - y_k \|\}$. We show that edge $(j,i)$ also is in $G_y$, which is true if $|V_y|<m$. Since $\taa_{ij}=1$, $|V_{x^0}|<m$. Thus, it is sufficient to show $V_y \subseteq V_{x^0}$. Let $k$ in $V_y$. Using the hypothesis of the lemma,
\begin{eqnarray*}
\|x_j^0-x_i^0\| &\ge& \|x_j^0-x_i^0\| + \| y_j - y_i - (x_j^0 - x_i^0)\| - \rho 
\ge \| y_j - y_i\| - \rho,
\end{eqnarray*}
where we have used the triangle inequality to obtained the last inequality. Using the same arguments,
\begin{eqnarray*}
\|x_k^0-x_i^0\| &\le& \|x_k^0-x_i^0\| - \| y_k - y_i - (x_k^0 - x_i^0)\| + \rho 
\le \| y_k - y_i\| + \rho.
\end{eqnarray*}
Combining the two previous inequalities, we obtain
\begin{eqnarray*}
\|x_j^0-x_i^0\| - \|x_k^0-x_i^0\| &\ge& \| y_j - y_i\| - \| y_k - y_i\| - 2\rho 
> - 2\rho,
\end{eqnarray*}
where we have used that $k \in V_y$ to obtain the last inequality. This shows that $k \in V_{x^0}$.
\end{proof}

Theorem~\ref{th:flocking-topo-main} assumes that $\rho$, the maximal authorized disturbance, is chosen so that $H_\rho$ has a spanning tree. We will exhibit a set of values of $\rho$ for which this hypothesis is satisfied. The highest value is denoted $\rho_{G_{x^0}}$ and is called the {\it robustness of graph $G_{x^0}$}. The corresponding graph $H_{\rho_{G_{x^0}}}$ is called the core subgraph of $G_{x^0}$ and is denoted $\KK(G_{x^0})$.

For $(j,i) \in G_{x^0}$, the robustness of interaction from $j$ to $i$ is defined as 
$$
s(j,i) = \frac{1}{2}\left(\|x_i^0 - x_{p(i)}^0\| - \|x_i^0 - x_j^0\|\right),
$$
where $p(i)$ is the index of the $m+1$-th closest agent of $i$ (consequently agent $p(i)$ does not influence $i$ in the initial interaction graph). Robustness $s(j,i)$ is chosen so that if initial distances between agents do not change more than $s(j,i)$, agent $j$ carries on influencing agent $i$. Consider a spanning tree in $G_{x^0}$. Denote $r$ its root. Let $i \in \NN \setminus \{r\}$ and $(i_0,i_2,\ldots,i_q)$ a path from $r$ to $i$ in $G_{x^0}$. We define the robustness of this path as
$$
s(i_0,i_2,\ldots,i_q) = \mmin_{0\le k \le q-1} s(i_k,i_{k+1}).
$$
Similarly, if the disturbance on the initial distance is smaller than robustness $s(i_0,i_2,\ldots,i_q)$ then the influence path $(i_0,i_2,\ldots,i_q)$ from $r$ to $i$ is preserved in the interaction graph. In order to preserve a spanning tree with root $r$, it is sufficient that one path is preserved from $r$ to $i$ for all $i \in \NN \setminus \{r\}$. We define the robustness from $r$ to $i$ as
$$
\rho_{ri} = \mmax_{(i_0,i_2,\ldots,i_q) \in Paths_{G_{x^0}}(r,i)} s(i_0,i_2,\ldots,i_q),
$$
where $Paths_{G_{x^0}}(r,i)$ is the set of paths from $r$ to $i$ in $G_{x^0}$. We then define the robustness of $r$ as a root of a spanning tree in $G_{x^0}$ as
$$
\rho_{r} = \mmin_{i\in \NN\setminus\{r\}} \rho_{ri}.
$$
Finally, we define the robustness of $G_{x^0}$ leading to the preservation of at least one spanning tree as
$$
\rho_{G_{x^0}} = \mmax_{r\in \NN} \rho_{r}.
$$
Then, denote  $\KK(G_{x^0}) = H_{\rho_{G_{x^0}}}$ the core subgraph of $G_{x^0}$.
These definitions allow to give the following proposition:
\begin{proposition}
Assume that the initial graph $G_{x^0}$ has a spanning tree. Then, for $\rho \in [0,\rho_{G_{x^0}}]$,
$H_\rho$ has a spanning tree whereas for $\rho > \rho_{G_{x^0}}$, $H_\rho$ has none.
\end{proposition}

\begin{proof}
We start with the proof of the first part of the proposition. Let $\rho \in [0,\rho_{G_{x^0}}]$. Denote $r\in \NN$ such that 
$$\rho_{G_{x^0}} = \rho_{r}.$$
We shall show that $H_{\rho}$ holds a spanning tree with root $r$. Let $i \in \NN \setminus \{r\}$. Denote
$(i_0,i_2,\ldots,i_q)$ a path from $r$ to $i$ in $G_{x^0}$ such that
$$
s(i_0,i_2,\ldots,i_q) = \rho_{ri}.
$$
Let $k \in \{0,\ldots,q-1\}$.
Then, we have
$$
s(i_k,i_{k+1}) \ge s(i_0,i_2,\ldots,i_q) = \rho_{ri} \ge \rho_r = \rho_{G_{x^0}} \ge \rho.
$$
Let us show that edge $(i_k,i_{k+1})$ is in $H_\rho$. To increase the readability, denote
$h = i_k$ and $l = i_{k+1}$.
We have,
$$
\|x_l^0 - x_{p(l)}^0\| - \|x_l^0 - x_h^0\| \ge 2 \rho.
$$
Denote $V = \{ u \in \NN\setminus \{h\} \;|\; \|x_l^0-x_h^0\| > \| x_l^0 - x_u^0 \|-2\rho\}$. Let $u$ in $V$. Let us show that $u$ is an in-neighbor of $l$ in $G_{x^0}$ (\ie $(u,l)$ is an edge of $G_{x^0}$):
\begin{eqnarray*}
\| x_l^0 - x_u^0 \| &<& 2\rho + \|x_l^0-x_h^0\| 
< (\|x_l^0 - x_{p(l)}^0\| - \|x_l^0 - x_h^0\|) + \|x_l^0-x_h^0\| 
< \|x_l^0 - x_{p(l)}^0\|.
\end{eqnarray*}
As expected, according to the definition of $p(l)$, $u$ belongs to the $m$ closest neighbors of $l$ in $G_{x^0}$ and by definition $u \neq h$, so $u$ can take at most $m-1$ values : $|V| <m$. This shows that $(h,l) \in H_\rho$. So, for all $i \in \NN \setminus \{r\}$, there exists a path from $r$ to $i$ in $H_\rho$. Thus, $H_\rho$ has a spanning tree with root $r$.

Turning to the second part of the result, assume that $\rho > \rho_{G_{x^0}}$. Then, for all $r\in \NN$,
$\rho_{r} < \rho$. It is possible to repeat the same type of reasoning we have used in the first part of the proof to obtain that for all $i \in \NN \setminus \{r\}$, and all paths $(i_0,i_2,\ldots,i_q)$ from $r$ to $i$ in $G_{x_0}$, there exists $k \in \{0,\ldots,q-1\}$ such that $(i_k,i_{k+1})$ is not in $H_\rho$. This is due to the fact that $p(i_{k+1})$, the $m+1$-th closest agent of $i_{k+1}$ in $G_{x^0}$ is in the neighbor set of $i_{k+1}$ in $H_\rho$.
\end{proof}

\subsection{Proofs}\label{sec:flocking-topo-preuves}

In this section, we show Lemma~\ref{l:deltauk1ttau-principal} as well as Theorem~\ref{th:flocking-topo-main}. Recall that proving Lemma~\ref{l:deltauk1ttau-principal} requires to show that the influence of $\UU_k$ on $\UU_{k+1}$ over time interval $[t,t+\tau]$ leads to the decrease of the diameter of $\UU_{k+1}$ as follows (equation~\ref{eq:rec}):
$$
\Delta_{\UU_{k+1}}(t+\tau) \leq \Delta_\NN(t_0) - c_k(\tau) (\Delta_{\NN}(t_0) - \Delta_{\UU_k}(t)).
$$
To obtain this result, we adopt the following reasoning:
\begin{description}
\item (i) (Lemma~\ref{l:dotdeltaUk1s}) The evolution of $\Delta_{\UU_{k+1}}$ at time $s \in [t,t+\tau]$ (\ie $\dot{\Delta}_{\UU_{k+1}}(s)$) is due to two terms:
\begin{itemize}
\item $\uak (\Delta_{\UU_k}(s) - \Delta_{\UU_{k+1}}(s))$ : $\Delta_{\UU_{k+1}}$ decreases thanks to agents in $\UU_k$, and
\item $(m-\uak) (\Delta_{\NN(t_0)}(s) - \Delta_{\UU_{k+1}}(s))$ : $\Delta_{\UU_{k+1}}$ increases because of agents not in $\UU_k$.
\end{itemize}
\item (ii) It would be possible to integrate the result from Lemma~\ref{l:dotdeltaUk1s}. However, we do not know the value of $\Delta_{\UU_k}(s)$ in the first term. To obtain it, we proceed as follows: since the increase of $\Delta_{\UU_k}$ is only a result of the influence of agents not in $\UU_k$, we obtain
$$
\dot{\Delta}_{\UU_k}(s) \le (m-\alpha_k) (\Delta_{\NN(t_0)}(s) - \Delta_{\UU_k}(s)).
$$
Integration then allows to bound the unknown value of $\Delta_{\UU_k}(s)$ as a function of $\Delta_{\UU_k}(t)$ (Lemma~\ref{l:deltaUk}).
\item (iii) The injection of the result from Lemma~\ref{l:deltaUk} in the one from Lemma~\ref{l:dotdeltaUk1s}, followed by integration yields Lemma~\ref{l:deltauk1ttau-principal}, as desired.
\end{description}


\begin{lemma}\label{l:deltaUk}
Let $k\in \{1,\ldots,D\}$. Let $t_0\ge 0$ some initial time. Let $t\geq t_0$ and $s>t$ some future time. Then, we have
\begin{equation}\label{eq:deltaUk}
\Delta_{\UU_k}(s) \leq \Delta_{\NN}(t_0) - e^{-(m-\uak) (s-t)}(\Delta_{\NN}(t_0) - \Delta_{\UU_k}(t)).
\end{equation}
\end{lemma}


Regarding the second step of the proof of Lemma~\ref{l:deltauk1ttau-principal}, we use Lemma~\ref{l:deltaUk} in order to estimate the contraction rate of $\Delta_{\UU_{k+1}}$.

\begin{lemma}\label{l:dotdeltaUk1s}
Let $t_0\ge 0$ some initial time and $s>t_0$ some future time. Then we have
\begin{equation}\label{eq-first-case}
\begin{array}{rclll}
\dot{\Delta}_{\UU_1}(s) &\leq& - \Delta_{\UU_1}(s) &+& m (\Delta_\NN(t_0) - \Delta_{\UU_1}(s))\\
&&&&\\
\text{ and for }& k \in& \{1,\ldots,D-1\},&&\\
\dot{\Delta}_{\UU_{k+1}}(s) &\leq& \underbrace{\uak (\Delta_{\UU_k}(s) - \Delta_{\UU_{k+1}}(s))} &+& \underbrace{(m-\uak) (\Delta_\NN(t_0) - \Delta_{\UU_{k+1}}(s))}.\\
&& \text{decrease of } \Delta_{\UU_{k+1}} && \text{increase of } \Delta_{\UU_{k+1}}\\
&& \text{due to } \UU_k && \text{due to the rest of the group }
\end{array}
\end{equation}

\end{lemma}


We merge this last result to the bound on $\Delta_{\UU_k}(s)$ for $s>t$ given by Lemma~\ref{l:deltaUk} to obtain a bound on $\Delta_{\UU_{k+1}}(t+\tau)$. We give the proofs of Lemmas~\ref{l:deltaUk},~\ref{l:dotdeltaUk1s} and~\ref{l:deltauk1ttau-principal} in this order.

In the proofs, we make use of the following lemma:
\begin{lemma}\label{lemma-sep}
Let $x$, $y$ and $z$ in $\R^n$ such that $\|x-y\| \geq \|x-z\|$. Then,
$$
\ps{y-x}{z-y} \leq 0.
$$
\end{lemma}
\begin{proof}
$\ps{y-x}{z-y} \leq \ps{y-x}{z-x+x-y} \leq \ps{y-x}{z-x} -\|x-y\|^2 \leq \|y-x\|(\|z-x\| -\|x-y\|) \leq 0$.
\end{proof}

\begin{proof}[Proof of Lemma~\ref{l:deltaUk}]
Let $t_0$, $t$ and $s$ successive times such that $t_0<t<s$. Let $k$ in $\{1,\ldots,D-1\}$. Let $i$ and $j$ two indices of $\UU_k$ maximizing the distance between velocities at time $s$ (\ie such that $\|v_i(s)-v_j(s)\|=\Delta_{\UU_k}(s)$). We bound above the derivative of $\|v_i-v_j\|$ using
\begin{equation}\label{eq:dot-vect}
	\dot{\overfence{\|v_i-v_j\|}} = \frac{\ps{\dot{v_i}-\dot{v_j}}{v_i-v_j}}{\|v_i-v_j\|}.
\end{equation}
We can write the numerator of the right-hand side of the equality in the following way :
\begin{equation}\label{eq-start}
\ps{\dot{v_i}-\dot{v_j}}{v_i-v_j} \: = \ssum_{h\in \NN}a_{ih}\ps{v_h-v_i}{v_i-v_j} + \ssum_{l\in \NN}a_{jl}\ps{v_l-v_j}{v_j-v_i}. \\
\end{equation}
In the two sums of this equation, $i$ and $j$ play symmetric roles. We start by studying the first sum; the result on the second one is then obtained through a similar reasoning. We can split the first sum in two parts: the influence from agents in $\UU_k$ and those from agents out of $\UU_k$, which gives
$$
\ssum_{h\in \NN}a_{ih}\ps{v_h-v_i}{v_i-v_j} = 
\ssum_{h\in \UU_k}a_{ih}\ps{v_h-v_i}{v_i-v_j} +
\ssum_{h\in \NN\setminus \UU_k}a_{ih}\ps{v_h-v_i}{v_i-v_j}.
$$
Applying Lemma~\ref{lemma-sep} with $(x,y,z):=(v_j,v_i,v_h)$, we obtain that the first sum of the previous equality is non-negative. This corresponds to the fact that agents in $\UU_k$ do not take part in a positive way in the increase of diameter $\Delta_{\UU_k}$.
Denote $h_{max}$ in $\NN$ such that
$$
\ps{v_{h_{max}}-v_i}{v_i-v_j} = \mmax_{h \in \NN} \ps{v_{h}-v_i}{v_i-v_j}.
$$
We then have
$$
\ssum_{h\in \NN}a_{ih}\ps{v_h-v_i}{v_i-v_j} \leq 
\ssum_{h\in \NN\backslash \UU_k}a_{ih} \ps{v_{h_{max}}-v_i}{v_i-v_j}.
$$
By definition of $h_{max}$, we have $\ps{v_{h_{max}}-v_i}{v_i-v_j} \geq 0$. 

So, according to Proposition~\ref{prop:topo-alphak}, 
\begin{equation}\label{eq:topo-hmax}
\ssum_{h\in \NN}a_{ih}\ps{v_h-v_i}{v_i-v_j}
\leq (m-\uak)\ps{v_{h_{max}}-v_i}{v_i-v_j}.
\end{equation}
Using a similar reasoning, we get
\begin{equation}\label{eq:topo-lmax}
\ssum_{l\in \NN}a_{jl}\ps{v_l-v_j}{v_j-v_i}
\leq (m-\uak)\ps{v_{l_{max}}-v_j}{v_j-v_i},
\end{equation}
where $l_{max} \in \NN$ is defined so that
$$
\ps{v_{l_{max}}-v_j}{v_j-v_i} = \mmax_{l \in \NN} \ps{v_{l}-v_j}{v_j-v_i}.
$$
Equations~(\ref{eq:topo-hmax}), (\ref{eq:topo-lmax}) and~(\ref{eq-start}) give
\begin{eqnarray*}
\ps{\dot{v_i}-\dot{v_j}}{v_i-v_j}
&\leq& (m-\uak)\left(\ps{v_{h_{max}}-v_{l_{max}}}{v_i-v_j} + \ps{v_i-v_j}{v_j-v_i}\right) \\
&\leq& (m-\uak)\left(\ps{v_{h_{max}}-v_{l_{max}}}{v_i-v_j} - \|v_i-v_j\|^2\right) \\
&\leq& (m-\uak)\left(\|v_{h_{max}}-v_{l_{max}}\|\cdot\|v_i-v_j\| - \|v_i-v_j\|^2\right) \\
&\leq& (m-\uak)\left(\Delta_{\NN}\|v_i-v_j\| - \|v_i-v_j\|^2\right).
\end{eqnarray*}
This result along with equation~(\ref{eq:dot-vect}) and $\Delta_\NN(s) \le \Delta_\NN(t_0)$ implies that
\begin{equation*}
\dot{\Delta}_{\UU_k}(s) \leq (m-\uak)(\Delta_{\NN}(t_0) - \Delta_{\UU_k}(s)).
\end{equation*}
After integrating over interval $[t,s]$, we obtain
\begin{equation}\label{eq-UUk}
\Delta_{\UU_k}(s) \leq \Delta_{\NN}(t_0) - e^{-(m-\uak) (s-t)}(\Delta_{\NN}(t_0) - \Delta_{\UU_k}(t)).
\end{equation}
\end{proof}


\begin{proof}[Proof of Lemma~\ref{l:dotdeltaUk1s}]
The proof of this lemma resembles the one of Lemma~\ref{l:deltaUk}.
Let $t_0$ and $s$ so that $0\le t_0\le s$. Let $k$ in $\{0,\ldots,D-1\}$.
Let $i,j \in \UU_{k+1}(s)$ distinct such that $\|v_i(s)-v_j(s)\|=\Delta_{\UU_{k+1}}(s)$.

Similarly to the previous proof, we bound above
\begin{equation*}
\ps{\dot{v_i}-\dot{v_j}}{v_i-v_j} = I_{\UU_k} + I_{\NN \backslash \UU_k},
\end{equation*}
with notation
\begin{eqnarray*}
I_{\UU_k} &=& \ssum_{h\in \UU_k}a_{ih}\ps{v_h-v_i}{v_i-v_j} + \ssum_{l\in \UU_k}a_{jl}\ps{v_l-v_j}{v_j-v_i},\\
I_{\NN \backslash \UU_k} &=& \ssum_{h\in \NN\backslash \UU_k}a_{ih}\ps{v_h-v_i}{v_i-v_j} + \ssum_{l\in \NN\backslash \UU_k}a_{jl}\ps{v_l-v_j}{v_j-v_i},
\end{eqnarray*}
where $I_{\UU_k}$ corresponds to the influence of agents in $\UU_k$ on the evolution of $\Delta_{\UU_{k+1}}$ and
$I_{\NN \backslash \UU_k}$ the one of agents out of $\UU_k$.
We start with providing an upper bound on $I_{\UU_k}$.

\noindent {\bf Upper bound on $I_{\UU_k}$}

We choose indices $h_{max}$ and $l_{max}$ in $\UU_k$ satisfying
\begin{eqnarray*}
\ps{v_{h_{max}}-v_i}{v_i-v_j} = \mmax_{h \in \UU_k} \ps{v_{h}-v_i}{v_i-v_j} \\
\text{and } \ps{v_{l_{max}}-v_j}{v_j-v_i} = \mmax_{l \in \UU_k} \ps{v_{l}-v_j}{v_j-v_i}.
\end{eqnarray*}
If $k=0$, then $\UU_k = \{r\}$ and $h_{max} = l_{max} = r$. We study two different cases.

\noindent {\it Case where $i=r$ or $j=r$.} \hspace{0.1cm}

Since the roles of $i$ and $j$ are symmetric, we can assume without loss of generality that $i=r$. Then, since $i\neq j$, $j \neq r$ and $j \in \NN_r^+$, so $a_{jr}=1$. This gives
\begin{equation}\label{eq:flo-topo-IUUk-k0}
I_{\UU_0} \leq a_{jr}\ps{v_r-v_j}{v_j-v_r} \leq -\|v_j-v_r\|^2 = -\|v_j-v_i\|^2.
\end{equation}

\noindent {\it Case where $i\neq r$ and $j\neq r$.} \hspace{0.1cm}

In this case, $a_{jr}=1$ and $a_{ir}=1$ and then
\begin{eqnarray*}
I_{\UU_0} &\leq& a_{ir} \ps{v_r-v_i}{v_i-v_j} + a_{ir}\ps{v_r-v_j}{v_j-v_i} \\
&\leq& \ps{v_r-v_i}{v_i-v_j} + \ps{v_r-v_j}{v_j-v_i} = \ps{v_j-v_i}{v_i-v_j} \\
&\leq& -\|v_j-v_j\|^2.
\end{eqnarray*}
We have obtained the same result in both cases.

\vspace{0.1cm}

If $k \geq 1$, we use Lemma~\ref{lemma-sep} and the fact that $\|v_i-v_j\| = \Delta_{\UU_{k+1}}$ to show that $\ps{v_h-v_i}{v_i-v_j} \leq 0$ and $\ps{v_l-v_j}{v_j-v_i} \leq 0$ for all $h,l$ in $\UU_k$. Thus,
\begin{equation*}
I_{\UU_k} \leq \uak (\ps{v_{h_{max}}-v_i}{v_i-v_j} + \ps{v_{l_{max}}-v_j}{v_j-v_i}).
\end{equation*}
The previous inequality can be rewritten as
\begin{eqnarray*}
I_{\UU_k} &\leq& \uak (\ps{v_{h_{max}}-v_{l_{max}}}{v_i-v_j} + \ps{v_i-v_j}{v_j-v_i})\\
&\leq& \uak (\Delta_{\UU_k} \|v_i-v_j\| - \|v_i-v_j\|^2).
\end{eqnarray*}
We now bound above $I_{\NN \backslash \UU_k}$.

\noindent {\bf Upper bound on $I_{\NN \backslash \UU_k}$}

For this part of the proof, it is possible to treat simultaneously the cases where $k=0$ and $k\ge 1$ by setting $\alpha_0 = 0$. Denote $h_{max}$ and $l_{max}$ in $\NN$ satisfying
\begin{eqnarray*}
\ps{v_{h_{max}}-v_i}{v_i-v_j} = \mmax_{h \in \NN} \ps{v_{h}-v_i}{v_i-v_j} \\
\ps{v_{l_{max}}-v_j}{v_j-v_i} = \mmax_{l \in \NN} \ps{v_{l}-v_j}{v_j-v_i}.
\end{eqnarray*}
Then,
\begin{eqnarray*}
I_{\NN \backslash \UU_k} &\leq& (m-\uak) (\ps{v_{h_{max}}-v_i}{v_i-v_j} + \ps{v_{l_{max}}-v_j}{v_j-v_i}) \\
&\leq& (m-\uak) (\ps{v_{h_{max}}-v_{l_{max}}}{v_i-v_j} + \ps{v_i-v_j}{v_j-v_i}) \\
&\leq& (m-\uak) (\Delta_\NN \|v_i-v_j\| - \|v_i-v_j\|^2),
\end{eqnarray*}
where we used $\ps{v_{h_{max}}-v_i}{v_i-v_j} \leq 0$ and $\ps{v_{l_{max}}-v_j}{v_j-v_i} \leq 0$ since it is always possible to choose $h_{max} := i$ and $l_{max}:=j$.

\vspace{0.5cm}

This result along with the bound on $I_{\UU_k}$ provided above grants, for $k=0$,
\begin{equation*}
\dot{\overfence{\|v_i(s)-v_j(s)\|}} \leq - \|v_j(s)-v_i(s)\| + m (\Delta_\NN(t_0) - \|v_i(s)-v_j(s)\|),
\end{equation*}
and for $k \geq 1$
\begin{equation*}
\dot{\overfence{\|v_i(s)-v_j(s)\|}} \leq \uak (\Delta_{\UU_k}(s) - \|v_j(s)-v_i(s)\|) + (m-\uak) (\Delta_\NN(t_0) - \|v_i(s)-v_j(s)\|).
\end{equation*}

According to the definition of $i$ and $j$, this rewrites to
\begin{equation*}
\begin{array}{rclll}
\dot{\Delta}_{\UU_1}(s) &\leq& - \Delta_{\UU_1}(s) &+& m (\Delta_\NN(t_0) - \Delta_{\UU_1}(s))\\
&&&&\\
\text{ and for } &k \geq 1,&&&\\
\dot{\Delta}_{\UU_{k+1}}(s) &\leq& 
\underbrace{\uak (\Delta_{\UU_k}(s) - \Delta_{\UU_{k+1}}(s))} &+& \underbrace{(m-\uak) (\Delta_\NN(t_0) - \Delta_{\UU_{k+1}}(s))}.\\
&& \text{decrease of } \Delta_{\UU_{k+1}} && \text{increase of } \Delta_{\UU_{k+1}}\\
&& \text{due to } \UU_k && \text{due to the rest of the group }
\end{array}
\end{equation*}
\end{proof}


\begin{proof}[Proof of Lemma~\ref{l:deltauk1ttau-principal}]
We first treat the case $k=0$. Lemma~\ref{l:dotdeltaUk1s} gives
\begin{eqnarray*}
\dot{\Delta}_{\UU_1}(s) &\leq& - \Delta_{\UU_1}(s) + m (\Delta_\NN(t_0) - \Delta_{\UU_1}(s)) \\
&\leq& (m+1) (\Delta_{\NN}(t_0) - \Delta_{\UU_1}(s)) - \Delta_{\NN}(t_0).
\end{eqnarray*}
An integration provides
\begin{eqnarray*}
\Delta_{\UU_1}(t+\tau)-\Delta_\NN(t_0) &\leq&
e^{-(m+1)\tau} \left(\Delta_{\UU_1}(t)-\Delta_\NN(t_0) - \int_{t}^{t+\tau} e^{(m+1)(s-t)} \Delta_{\NN}(t_0) 
ds\right)\\
&\leq& e^{-(m+1)\tau} \left(\Delta_{\UU_1}(t)-\Delta_\NN(t_0) - \frac{1}{m+1} (e^{(m+1)\tau}-1) \Delta_{\NN}(t_0) \right)\\
&\leq& e^{-(m+1)\tau} \left(\Delta_{\UU_1}(t)-\Delta_\NN(t_0) - \frac{1}{m+1} (e^{(m+1)\tau}-1) \Delta_{\NN}(t_0) \right).\\
\end{eqnarray*}
Since $\Delta_{\UU_1}(t) \leq \Delta_\NN(t_0)$, we have
\begin{equation*}
\Delta_{\UU_1}(t+\tau) 
\leq \Delta_\NN(t_0) - \frac{1}{m+1} (1 - e^{-(m+1)\tau}) \Delta_{\NN}(t_0).
\end{equation*}
Using notation $c_0(\tau) = \frac{1}{m+1} (1 - e^{-(m+1)\tau})$, we obtained the expected result. 

For $k\geq 1$, we replace $\Delta_{\UU_k}(s)$ in equation~(\ref{eq-first-case}) of Lemma~\ref{l:dotdeltaUk1s} by its upper bound given by equation~(\ref{eq:deltaUk}) of Lemma~\ref{l:deltaUk} to obtain
\begin{eqnarray*}
\dot{\Delta}_{\UU_{k+1}}(s) &\leq& 
\uak \left(\Delta_{\NN}(t_0) - e^{-(m-\uak) (s-t)}\left(\Delta_{\NN}(t_0) - \Delta_{\UU_k}(t)\right) - \Delta_{\UU_{k+1}}(s)\right) 
+ (m-\uak) \left(\Delta_\NN(t_0) - \Delta_{\UU_{k+1}}(s)\right) \\
&\leq& m \left(\Delta_{\NN}(t_0) - \Delta_{\UU_{k+1}}(s)\right) - \uak e^{-(m-\uak) (s-t)}\left(\Delta_{\NN}(t_0) - \Delta_{\UU_k}(t)\right).
\end{eqnarray*}
So,
$$
\frac{d}{ds} \left(\Delta_{\UU_{k+1}}(s) - \Delta_{\NN}(t_0)\right) \leq
-m \left(\Delta_{\UU_{k+1}}(s) - \Delta_{\NN}(t_0)\right) - \uak e^{-(m-\uak) (s-t)}\left(\Delta_{\NN}(t_0) - \Delta_{\UU_k}(t)\right).
$$
An integration gives
\begin{eqnarray*}
&&\Delta_{\UU_{k+1}}(t+\tau)-\Delta_\NN(t_0) \\ 
&&\leq e^{-m\tau} \left(
\Delta_{\UU_{k+1}}(t)-\Delta_\NN(t_0)
- \int_{t}^{t+\tau} e^{m(s-t)} \uak e^{-(m-\uak) (s-t)}(\Delta_{\NN}(t_0) - \Delta_{\UU_k}(t)) ds\right)\\
&&\leq
e^{-m\tau} \left(
\Delta_{\UU_{k+1}}(t)-\Delta_\NN(t_0)
- \uak (\Delta_{\NN}(t_0) - \Delta_{\UU_k}(t)) \int_{t}^{t+\tau} e^{\uak(s-t)} ds\right)\\
&&\leq
e^{-m\tau} \left(
\Delta_{\UU_{k+1}}(t)-\Delta_\NN(t_0)
- (\Delta_{\NN}(t_0) - \Delta_{\UU_k}(t)) (e^{\uak\tau}-1)\right)\\
&&\leq
e^{-m\tau} (\Delta_{\UU_{k+1}}(t)-\Delta_\NN(t_0)) 
- (\Delta_{\NN}(t_0) - \Delta_{\UU_k}(t))(e^{-(m-\uak) \tau}-e^{-m\tau}).
\end{eqnarray*}
Since $\Delta_{\UU_{k+1}}(t) \leq \Delta_\NN(t_0)$, we can write
\begin{equation*}
\Delta_{\UU_{k+1}}(t+\tau) \leq \Delta_\NN(t_0) - c_k(\tau) (\Delta_{\NN}(t_0) - \Delta_{\UU_k}(t)),
\end{equation*}
with $c_k(\tau) =  e^{-(m-\uak) \tau}(1-e^{-\uak\tau}) \geq 0$.
\end{proof}

We now give the proof of the main theorem.

\begin{proof}[Proof of Theorem~\ref{th:flocking-topo-main}]
We show this result by contradiction.
Assume that there exists some time $t>0$ for which $G(t)$ does not satisfy hypothesis~\ref{as:alpha-preserv} (\ie there exists $k$ in $\{1,\ldots,D-1\}$, for which $\uak$ is not a valid bounds in $G(t)$). Denote $t^*$ the lower bound on such a time $t$. If $t^*>0$, we have, according to Theorem~\ref{th:flocking-topo-contraction}, for all $t$ in $[0,t^*[$
\begin{equation*}
\Delta_\NN(t) \leq (1-\tilde{c})^h \Delta_\NN(0),
\end{equation*}
where $h$ is such that $h\leq t/\tilde{T} < h+1$.
Let $j$ and $i$ in $\NN$. Then
\begin{eqnarray*}
\|(x_j(t^*) - x_i(t^*)) - (x_j(0) - x_i(0))\| & \leq & \| \int_0^{t^*} (v_j(s) - v_i(s)) ds\| \\
\leq  \int_0^{t^*} \| v_j(s) - v_i(s) \| ds 
& \leq & \int_0^{t^*} \Delta_\NN(s) ds 
 \leq  \ssum_{h=0}^k \int_{h\tilde{T}}^{(h+1)\tilde{T}} \Delta_\NN(s) ds \\
\end{eqnarray*}
where $k$ is such that $k\leq t^*/\tilde{T} < k+1$.
This calculus results in
\begin{eqnarray*}
\|(x_j(t^*) - x_i(t^*)) - (x_j(0) - x_i(0))\|
& \leq & \tT \ssum_{h=0}^k (1-\tc)^h \Delta_\NN(0) ds \\
 <  \tT \ssum_{h=0}^{+\infty} (1-\tc)^h \Delta_\NN(0) 
& < & \frac{\tT}{\tc} \Delta_\NN(0). \\
\end{eqnarray*}
Using the bound on $\Delta_\NN(0)$ in the assumption of the theorem, we have
\begin{equation*}
\|(x_j(t^*) - x_i(t^*)) - (x_j(0) - x_i(0))\| < \rho.
\end{equation*}
By continuity of trajectory $x$, there exists $\varepsilon>0$ such that for all $t \in [t^*, t^*+\varepsilon]$
\begin{equation*}
\|(x_j(t) - x_i(t)) - (x_j(0) - x_i(0))\| \le \rho.
\end{equation*}
If $t^*=0$, the continuity of $x$ gives the same property.
According to Lemma~\ref{lemma:preserv}, we have for all $t \in [0, t^*+\varepsilon]$, $H_\rho \subseteq G(t)$ and thus according to Proposition~\ref{prop:topo-alphak}, $G(t)$ satisfies Hypothesis~\ref{as:alpha-preserv} (the $\uak$ remain valid bounds for $G(t)$). This leads to a contradiction with the definition of $t^*$. Thus, 
$G(t)$ satisfies Hypothesis~\ref{as:alpha-preserv} for all $t\ge 0$. Then, we can use the reasoning of the proof to deduce that $H_\rho \subseteq G(t)$ for all $t\ge 0$. Consequently, Theorem~\ref{th:flocking-topo-contraction} shows that diameter $\Delta_{\NN}(t)$ converges toward $0$ when $t$ goes to $+\infty$. Velocity alignment is reached asymptotically.
\end{proof}

\section{Numerical analysis}

\subsection{Optimization of the bound}\label{sec:flocking-topo-optimisation-de-la-borne}

In this section, we explicit the method to maximize the bound $\frac{\tc}{\tT} \rho$ given in Theorem~\ref{th:flocking-topo-main}. 

First, assume that the maximal authorized disturbance $\rho$ and root $r$ are fixed. Sets $\UU_k$ and $\SS_k$ for $k \in \{0,\ldots,D\}$ and values $\uak$ for $k \in \{1,\ldots,D-1\}$ are thus known. We explicit the sequence $\tbtau = (\tak{0},\ldots,\tak{D-1})$ for which ratio $\frac{c}{T}$ is maximum. We will then discuss the way of choosing $\rho$ and $r$.

\begin{proposition}
The value $\tT$ is the solution to the following equation with unknown $T$:
\begin{equation}\label{eq:flo-topo-equation-en-T-finale}
e^T = \left((m+1)T+1\right)^{\frac{1}{m+1}} \pprod_{k=1}^{D-1} \left(\frac{mT+1}{(m-\uak)T+1}\right)^{\frac{1}{\uak}}.
\end{equation}
Moreover, the optimal sequence $\tbtau = (\tak{0},\ldots,\tak{D-1})$ is defined by
$$
\tak{0} = \frac{1}{m+1}\ln\left((m+1)\tT+1\right),
$$
and for $k \in \{1,\ldots,D-1\}$,
$$
\tak{k} = \frac{1}{\uak}\ln\left(\frac{m\tT+1}{(m-\uak)\tT+1}\right).
$$
\end{proposition}

\begin{proof}
Let $k \in \zeroDmun$. We have
\begin{eqnarray*}
\deriv{\frac{c}{T}}{\tau_k} &=& \frac{\deriv{c}{\tau_k} \cdot T - c}{T^2} 
= \frac{ \left( \deriv{c_k}{\tau_k}(\tau_k) T - c_k(\tau_k) \right) \pprod_{h=0,h\neq k}^{D-1} c_h(\tau_h) }{T^2}.
\end{eqnarray*}
Consequently, when for all $h \in \zeroDmun$, $\tau_h >0$, we have
$$
\deriv{\frac{c}{T}}{\tau_k} = 0 \;\; \iff \;\;
\deriv{c_k}{\tau_k}(\tau_k) T - c_k(\tau_k) = 0.
$$
The calculus gives, for $k\ge 1$,
\begin{eqnarray*}
\deriv{c_k}{\tau_k}(\tau_k) &=& -m e^{-m\tau_k} (e^{\uak \tau_k} -1 ) + \uak e^{-m\tau_k} e^{\uak \tau_k} \\
&=& e^{-m\tau_k} (m-(m-\uak) e^{\uak \tau_k} ).
\end{eqnarray*}
Thus, 
\begin{equation}\label{eq:flo-topo-tauk}
\begin{array}{lll}
\deriv{\frac{c}{T}}{\tau_k} = 0 
&\Leftrightarrow& (m-(m-\uak) e^{\uak \tau_k} ) T = e^{\uak \tau_k} -1\\
&\Leftrightarrow& mT+1 = e^{\uak \tau_k}(1+(m-\uak)T) \\
&\Leftrightarrow& \tau_k = \frac{1}{\uak}\ln\left(\frac{mT+1}{(m-\uak)T+1}\right).
\end{array}
\end{equation}
In a similar way, for $k=0$,
\begin{eqnarray*}
\deriv{c_0}{\tau_0}(\tau_0) &=& e^{-(m+1)\tau_0}.
\end{eqnarray*}
Thus, 
\begin{equation}\label{eq:flo-topo-tau0}
\begin{array}{lll}
\deriv{\frac{c}{T}}{\tau_0} = 0 &\Leftrightarrow& e^{-(m+1)\tau_0} T = \frac{1}{m+1}(1-e^{-(m+1) \tau_0}) \\
&\Leftrightarrow& e^{-(m+1)\tau_0}((m+1)T+1) = 1\\
&\Leftrightarrow& \tau_0 = \frac{1}{m+1}\ln((m+1)T+1).
\end{array}
\end{equation}
Using the definition of $T$, for $k \in \zeroDmun$, $\deriv{\frac{c}{T}}{\tau_k} = 0$ if and only if
\begin{equation}\label{eq:flo-topo-equation-en-T}
T = \frac{1}{m+1}\ln((m+1)T+1) + \ssum_{k=1}^{D-1} \frac{1}{\uak}\ln\left(\frac{mT+1}{(m-\uak)T+1}\right),
\end{equation}
on the one hand, and on the other hand the $\tau_k$ are defined in function of $T$ according to the equality just obtained (equations~(\ref{eq:flo-topo-tauk}) and~(\ref{eq:flo-topo-tau0})). Equation~(\ref{eq:flo-topo-equation-en-T}) rewrites to
$$
e^T = ((m+1)T+1)^{\frac{1}{m+1}} \pprod_{k=1}^{D-1} \left(\frac{mT+1}{(m-\uak)T+1}\right)^{\frac{1}{\uak}},
$$
which corresponds to equation~(\ref{eq:flo-topo-equation-en-T-finale}) of the proposition. It remains to show that the equation has a unique solution but the trivial solution $0$. We show this for the equivalent equation (\ref{eq:flo-topo-equation-en-T}). The right-hand term is null if $T=0$. Computing the second derivative of this term shows it is concave in the variable $T$. Moreover, this term is equivalent to $(1 + \sum_{k=1}^{D-1} \frac{m}{\uak})T$ when $T$ approaches $0$, whose increase is strictly faster than the identity function. Finally, this term is equivalent to $\frac{\ln((m+1)T)}{m+1}$ when $T$ goes to $+\infty$, whose increase is strictly slower than the identity function. Consequently, this term has exactly two intersections with the identity function: one for $T=0$ and the other for some $T>0$, which is $\tT$. Ratio $\frac{c}{T}$ is non-negative, and it is null if $\tau_k = 0$ and when $\tau_k \rightarrow +\infty$. Consequently, the obtained optimum is necessarily the expected maximum.
\end{proof}

Now, assume that only the disturbance $\rho$ is set. One possibility to find $r$ giving the maximal ratio $\frac{\tc}{\tT}$ is to test all nodes in $\NN$. Alternatively, we may expect that a tree with minimal depth allows for a faster propagation and yields the maximal value $\frac{\tc}{\tT}$.

Finally, in order to choose $\rho$, a finite number of tests of distinct $\rho$ values suffice to obtain the best. These values corresponds to the distinct values of the $\left| \|x_i^0 - x_k^0\| - \|x_i^0 - x_j^0\| \right|$ for  $i,j,k \in \NN$ distinct resulting in a connected graph $H_\rho$.

\subsection{Simulations}

In this section, we carry out simulations to illustrate Theorem~\ref{th:flocking-topo-main}.

Consider 4 agents moving according to system (\ref{sys:alignement-en-vitesse-topo}) where each agent is influenced by $m=1$ agent. Initially, regarding the positions, we have $x^0_1 = (0,2.7)$, $x^0_2 = (0,2.5)$, $x^0_3 = (0,1.5)$, $x^0_4 = (0,0)$ and regarding the velocities, $w^0_1 = (0,1)$, $w^0_2 = (0,1)$, $w^0_3 = (0,-1)$ and $w^0_4 = (0,1)$. We then choose $v_i^0 = \alpha w_i^0+(c,0)$ where $c$ is a constant parameter which does not influence the velocity alignment and is only used for visualization purpose and $\alpha$ is a parameter used to modify the initial velocity diameter. Such a configuration gives the following initial interaction graph matrix: $(0,1,0,0;1,0,0,0;0,1,0,0;0,0,1,0)$, written \textit{\`a la} Matlab.
Consequently, the subgraph $H_\rho$ having a spanning tree in the initial interaction graph is the initial graph itself. The maximum robustness $\rho$ allowing to preserve $G_{x^0}$ (\ie such that $H_\rho = G_{x^0}$) is half of the distance between agents $2$ and $3$ : $\rho = 0.25$. Regarding the hierarchical structure corresponding to $G_{x^0}$, we have $D = 2$ the depth of the tree, root $r = 2$, $\SS_0 = \{2\}$, $\SS_1 = \{1,3\}$ and $\SS_2 = \{4\}$. These sets are associated to flow $\alpha_1 = 1$. We obtain the bound
$
\frac{\tilde{c}}{\tilde{T}} \rho = 0.0351,
$
using comments from section~\ref{sec:flocking-topo-optimisation-de-la-borne}. Figure~\ref{fig:flocking-topo-4-agents} presents the simulation results for various values of the initial velocity diameter (modified using parameter $\alpha$). When the initial diameter is equal to the bound $\frac{\tilde{c}}{\tilde{T}} \rho$, the interaction graph preserves the spanning tree (and here $G_{x^0}$) through time, as guaranteed by Theorem~\ref{th:flocking-topo-main}. On the opposite, when the initial velocity diameter is too large (\eg $\Diam(0) = 13\frac{\tilde{c}}{\tilde{T}} \rho$), agent $3$ approaches agent $4$ so that distance between $3$ and $4$ becomes smaller than the one between $3$ and $2$. As a consequence, the interaction graph gets disconnected and the velocity alignment can never be reached. On the other hand, notice that when $\Diam(0) = 10\frac{\tilde{c}}{\tilde{T}} \rho$, case where Theorem~\ref{th:flocking-topo-main} does not allow to conclude, the group converges toward velocity alignment. This illustrates the fact that Theorem~\ref{th:flocking-topo-main} only provides a sufficient condition but not a necessary one.

\begin{figure}[!htbp]
\begin{center}
\includegraphics[scale=0.35,trim=4cm 9cm 4cm 9cm,keepaspectratio]{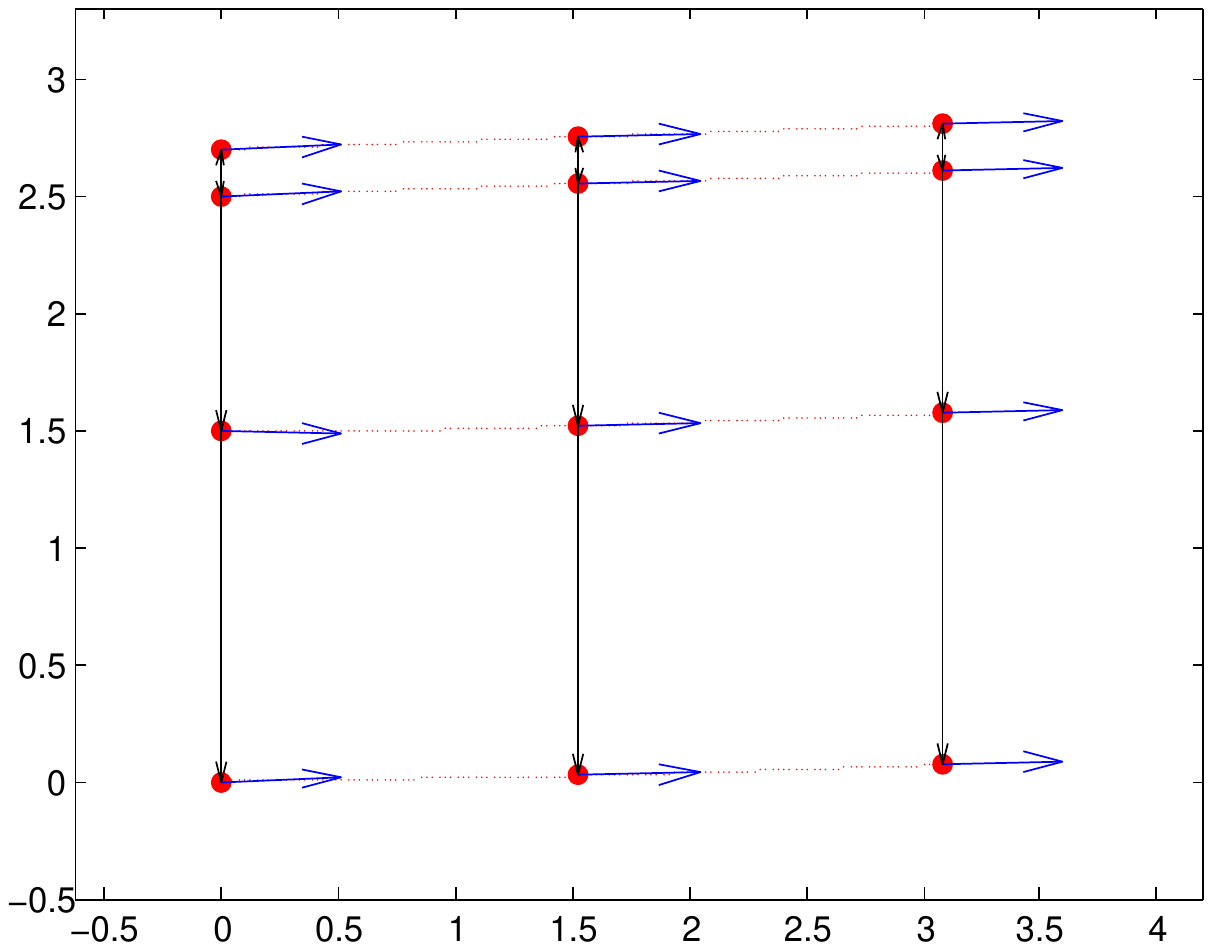}
\includegraphics[scale=0.35,trim=4cm 9cm 4cm 9cm,keepaspectratio]{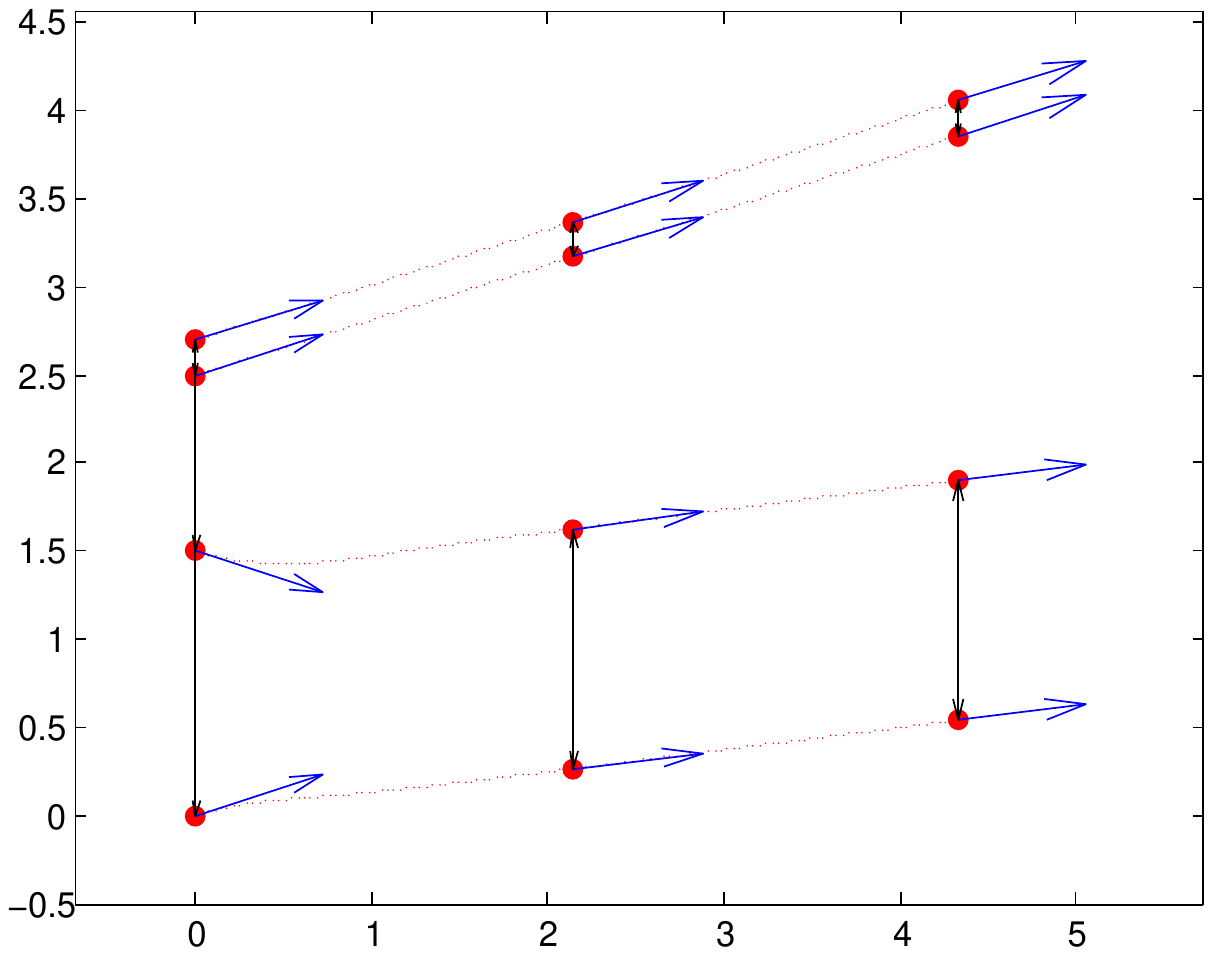}\\
\includegraphics[scale=0.35,trim=4cm 9cm 4cm 9cm,keepaspectratio]{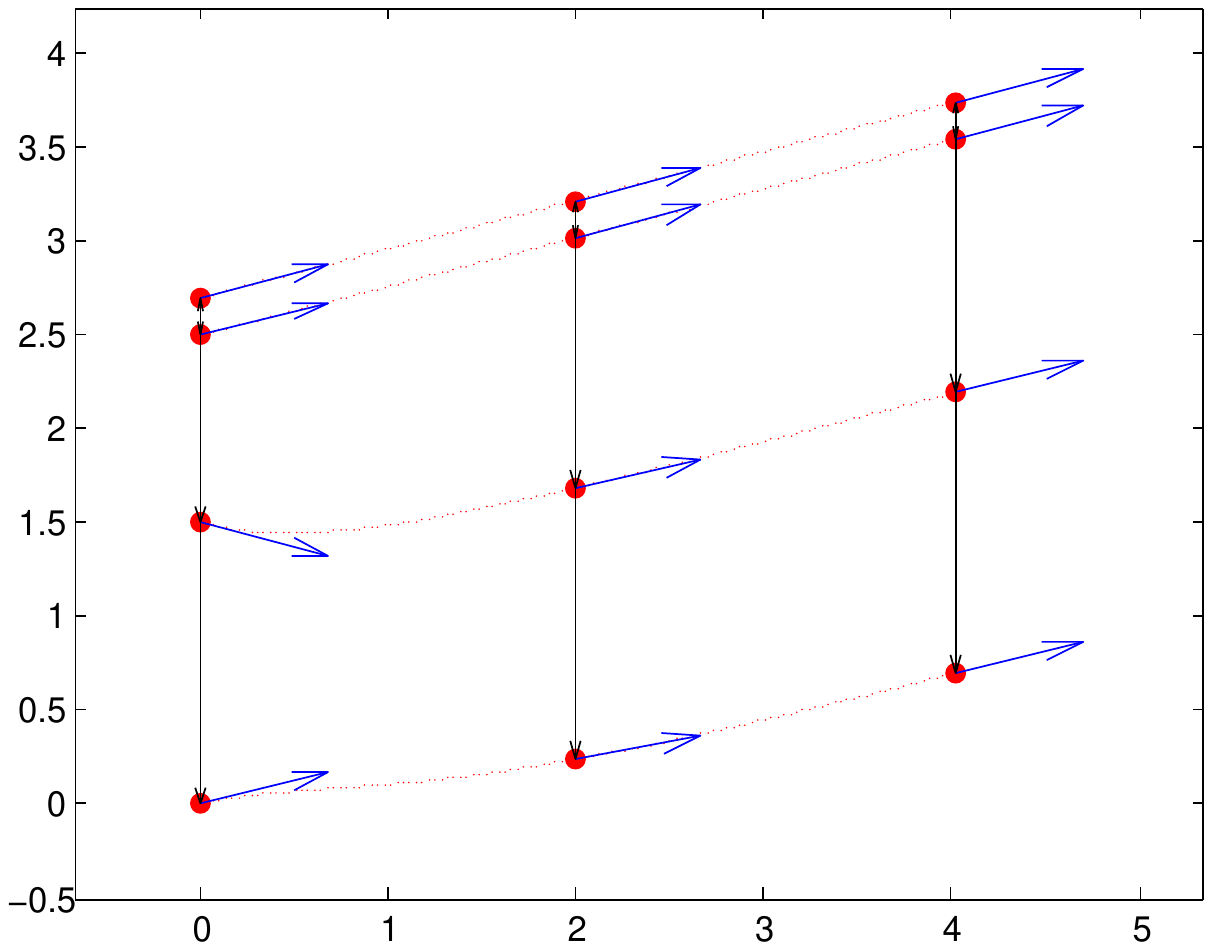}
\end{center}
\caption{\label{fig:flocking-topo-4-agents} The top-left figure shows an instance of the trajectory when $\Diam(0) = \frac{\tilde{c}}{\tilde{T}} \rho$. The top-right figure is an instance of the trajectory when $\Diam(0) = 13\frac{\tilde{c}}{\tilde{T}} \rho$ and the bottom figure corresponds to $\Diam(0) = 10\frac{\tilde{c}}{\tilde{T}} \rho$. The figure shows the trajectories of the $4$ agents as well as their positions, velocities and interactions for $3$ given times. The red dots correspond to the agents' positions. The blue arrows are the agents' velocities. The black arrows between agents represent the interactions.}
\end{figure}

\section{Conclusion}

In this paper, we have considered a multi-agent system consisting of mobile agents with second-order dynamics, where the communication network is determined by the so-called topological rule. Our approach adapts ideas from Martin and Girard~\cite{Martin2010}, Moreau~\cite{Moreau2004} and Angeli and Bliman~\cite{Angeli2009}. It links the preservation of a hierarchical structure in the communication network to the speed of convergence towards consensus. 

We have established a sufficient condition for velocity alignment depending on the initial positions and velocities of the agents only. Our main contribution has been to provide a new convergence rate toward consensus valid for all time for the continuous-time consensus system. Our main theoretical result states that flocking occurs whenever the initial velocity diameter is smaller than a threshold (which is a function of the robustness of some subgraph of the initial interaction graph). This result allowed us to derived practical bounds for flocking. Finally, we have illustrated the validity of our approach through simulations. The main interest of our approach is the possibility of ensuring flocking a priori. The condition can be easily verified through numerical computation.

For future work, we plan to improve the tightness of the bound by taking into account two facts:
we will relate velocities with positions because two agents with opposite velocities have more chance to agree on their velocities if they point toward each other, than if they point away from each other. Also, a subgroup of agents with high density is intuitively more inclined to agree on their velocities than
a subgroup of low connectivity. Thus, agents belonging to a highly connected local neighborhood should be allowed
higher initial velocities (see for instance~\cite{FiacchiniMorarescu2012}).
Finally, our general theoretical result providing a convergence rate toward consensus may be applied to other non-symmetric communication rules.

\bibliographystyle{alpha}
{\small \bibliography{references}}



\end{document}